\documentclass[a4]{amsart}

\input xypic 
\input xy 
\xyoption{all} 
\usepackage{amssymb} 
\usepackage{bbm}
\usepackage{tikz-cd}

\usepackage{float}
\usepackage{graphicx}
\usepackage{accents}
\usepackage{amsthm}

\usepackage[bookmarksnumbered, bookmarksopen, backref]{hyperref}

\newtheorem{theorem}{Theorem}[section]

\newtheorem{lemma}[theorem]{Lemma}

\theoremstyle{definition}

\newtheorem{remark}[theorem]{Remark}
\newtheorem{example}[theorem]{Example}

\newtheorem*{remark*}{Remark}
\newtheorem{question}[theorem]{Question}

\newcounter{RomanNumber}
\newcommand{\MyRoman}[1]{\setcounter{RomanNumber}{#1}\Roman{RomanNumber}}
 
\newcommand{\vees}[1]{\ensuremath{\mathop{\bigvee}\limits_}}
\newcommand{\bigopluss}[1]{\ensuremath{\mathop{\bigoplus}\limits_}}

\newcounter{bean}

\newcommand{\larrow}{\relbar\!\!\relbar\!\!\rightarrow}
\newcommand{\llarrow}{\relbar\!\!\relbar\!\!\larrow}
\newcommand{\lllarrow}{\relbar\!\!\relbar\!\!\llarrow}

\begin{document}
\begin{sloppypar}

\title[Sphere bundles after looping]{Sphere bundles over $4$-manifolds are trivial after looping} 

\author{Ruizhi Huang} 
\address{State Key Laboratory of Mathematical Sciences \& Institute of Mathematics, Academy of Mathematics and Systems Science, 
   Chinese Academy of Sciences, Beijing 100190, China} 
\email{huangrz@amss.ac.cn} 
   \urladdr{https://sites.google.com/site/hrzsea/}
   \thanks{}

\date{}


\begin{abstract} 
We show that except two special cases, the sphere bundle of a vector bundle over a simply connected $4$-manifold splits after looping. In particular, this implies that though there are infinitely many inequivalent sphere bundles of a given rank over a $4$-manifold, the loop spaces of their total manifolds are all homotopy equivalent.
\end{abstract}

\maketitle

\section{Introduction}
Sphere bundles over $4$-manifolds are classical and important in topology. 
There are many famous results especially when the base manifold is the $4$-sphere.
For instance, in differential topology, Milnor \cite{Mil56} found his exotic $7$-sphere as the total manifold of a $3$-sphere bundle over the $4$-sphere. The total manifolds of general $S^3$-bundles over $S^4$ were eventually classified by Crowley and Escher \cite{CE03} in various categories. In homotopy theory, in their remarkable work \cite{JW54, JW55}, James and Whitehead studied deeply the fibrewise homotopy classification of general sphere bundles over spheres.
In contrast, there are much fewer investigations on sphere bundles over general $4$-manifolds, the topology of which should be much harder.

On the other hand, loop space decompositions provide powerful tools for studying the homotopy properties of manifolds. For example, Beben and Theriault \cite{BT14} studied the loop decompositions of $(n-1)$-connected $2n$-manifolds, while Beben and Wu \cite{BW15} and Huang and Theriault \cite{HT22} investigated those of $(n-1)$-connected $(2n+1)$-manifolds. Their loop decompositions allow for the computation of the homotopy groups of these manifolds\footnote{Also studied by Samik Basu and Somnath Basu \cite{Bas19, BB18} from a different perspective.} from those of spheres and Moore spaces. Furthermore, a theoretical method of loop space decomposition was developed by Beben and Theriault \cite{BT22} and was significantly expanded by Theriault in \cite{The24a}. This method has been applied to investigate the homotopy of various geometric structures, including connected sums \cite{Che22, Che23a, Che23b, HT23, The24a}, surgeries \cite{HT23}, open books \cite{HT24a}, blow ups \cite{HT24b} and certain polyhedral products \cite{ST24}. It is also crucial in the study of inert top cell attachments of orientable closed manifolds \cite{Hua24, The24b}. For a comprehensive introduction to loop space decompositions of manifolds and their applications, see the recent book by Huang and Theriault \cite{HT25}.

\medskip

\noindent{\bf Main Results.}\quad In this paper, we study the homotopy theory of sphere bundles over general simply connected $4$-manifolds from the loop viewpoint. The topology of circle bundles over $4$-manifolds has been studied by \cite{DL05}, from which Beben and Theriault \cite{BT14} have determined their loop homotopy types. 
The homotopy theory of $2$-sphere bundles was studied by the author recently \cite{Hua22}.
Our main theorem below determines the loop homotopy types of sphere bundles over simply connected $4$-manifolds for almost all other cases. In particular, it illustrates that though there are infinitely many inequivalent sphere bundles of a given rank over a $4$-manifold, the loop spaces of their total manifolds are all homotopy equivalent.

\begin{theorem}\label{decomthm}
Let $d$ and $n$ be nonnegative integers such that $n\geq 2$, and $(d, n)\neq (0, 2)$, $(0,3)$. Let $N$ be a simply connected closed $4$-manifold such that $H^2(N;\mathbb{Z})\cong \mathbb{Z}^{\oplus d}$. Let
\[
S^{n}\stackrel{}{\longrightarrow} M\stackrel{}{\longrightarrow} N
\]
be the sphere bundle of a rank $(n+1)$ vector bundle over $N$. Then the sphere bundle splits after looping to give a homotopy equivalence
\[
\Omega M\simeq \Omega S^{n}\times \Omega N.
\]
Moreover, 
\begin{itemize}
\item[(1).] if $d=0$,
$
\Omega M\simeq S^3\times \Omega S^{n}\times \Omega S^7;
$
\item[(2).] if $d=1$, 
   $ 
   \Omega M\simeq S^1\times \Omega S^{n}\times \Omega S^5;
   $ 
\item[(3).] if $d\geq 2$,
  $
   \Omega M\simeq S^1\times \Omega S^{n}\times\Omega (S^2\times S^3)\times  \Omega\big(J\vee(J\wedge\Omega (S^2\times S^3))\big),
 $
where $J=\mathop{\bigvee}\limits_{i=1}^{d-2}(S^2\vee S^3)$.
   \end{itemize}
\end{theorem}
The theorem will be proved by the three indicated cases in Sections \ref{sec: d>1}, \ref{sec: d=1} and \ref{sec: d=0} respectively. 
As an immediate consequence of Theorem \ref{decomthm}, we have an isomorphism of homotopy groups
\[
\pi_\ast(M)\cong \pi_\ast(S^n)\oplus \pi_\ast(N),
\]
where $\pi_\ast(N)$ can be expressed as a direct sum of homotopy groups of spheres by the further decompositions in the theorem (cf. \cite[Introduction]{BT14}). 
Note that this isomorphism can not be obtained directly from the long exact sequence of the homotopy groups of the sphere bundle. 

Similarly, we have an isomorphism of homology groups
\[
H_\ast(\Omega M;\mathbb{Z})\cong H_\ast(\Omega S^n;\mathbb{Z})\otimes H_\ast(\Omega N;\mathbb{Z}),
\]
where $H_\ast(\Omega S^n;\mathbb{Z})$ is a rank $1$ tensor algebra by the Bott-Samelson theorem \cite{BS53}, and $H_\ast(\Omega N;\mathbb{Z})$ was computed by Sa. Basu and So. Basu \cite[Theorem 4.1]{BB18} as a quadratic algebra. Further, as Example \ref{Hnonex} illustrates, in general the homology of the sphere bundle does not split before looping
\[
H_\ast(M;\mathbb{Z})\not\cong H_\ast( S^n;\mathbb{Z})\otimes H_\ast(N;\mathbb{Z}).
\]
It is worth noting that a homology isomorphism arising from a loop space decomposition does not, in general, preserve the multiplicative structure. Nonetheless, determining the homology algebra of loop spaces remains a fundamental problem in algebraic topology. 

Note that there are two cases not considered in the theorem. These are the sphere bundles of the form
\[
S^{n}\stackrel{}{\longrightarrow} M\stackrel{}{\longrightarrow} S^4
\]
with $n=2$ or $3$. When the bundle does not have a cross section, partial results on the loop homotopy of $M$ were obtained by the author \cite{Hua22} for $n=2$ and by Theriault and the author \cite{HT22} for $n=3$. Based on the results there, it can be realized that these two cases are quite different from the other cases when $n\geq4$. In general, when $n=2$ or $3$ the sphere bundle does not split after looping, and there are torsions in the loop homology of $M$. 
More details can be found in a summary in Section \ref{sec: d=0}.

\medskip

\noindent{\bf Future Questions.}\quad There are several directions for further exploration. 

Firstly, the method presented in this paper is not applicable to non-simply connected $4$-manifolds. It would be interesting to explore new methods for investigating the loop homotopy of sphere bundles over non-simply connected $4$-manifolds.
\begin{question}
Let $N$ be a non-simply connected $4$-manifold with nilpotent fundamental group. For a positive integer $n$, classify the loop homotopy type of the sphere bundles of rank $(n+1)$ vector bundles over $N$. 
\end{question}

Furthermore, the main result of the paper implies that the function
\[
\Omega S: \frac{\{\text{rank $(n+1)$ vector bundles over $N$}\}}{\cong}\llarrow  \frac{\{\text{homotopy fibrations over $\Omega N$}\}}{\simeq},
\]
which sends a bundle to the loop of its sphere bundle, is trivial for most simply connected $4$-manifolds $N$. This leads to the following general question.
\begin{question}
Let $n$ be a positive integer. 
\begin{itemize}
\item[(1).] For which manifolds $N$ is the function $\Omega S$ trivial?
\item[(2).] For which manifolds $N$ does the function $\Omega S$ have a finite image?
\end{itemize}
\end{question}

Additionally, it is well-known that characteristic classes serve as obstructions to trivializing bundles, and can provide complete invariants in certain cases. Inspired by this, it is possible to study the loop homotopy of sphere bundles from the perspective of characteristic classes.  
\begin{question}
Let $\xi$ and $\eta$ be two vector bundles over a fixed manifold $N$. If $\Omega S(\xi)=\Omega S(\eta)$, what is the relation between the characteristic classes of $\xi$ and $\eta$?
\end{question}

\bigskip

\noindent{\bf Acknowledgement} The author was supported in part by the National Natural Science Foundation of China (Grant nos. 12331003 and 12288201), the National Key R\&D Program of China (No. 2021YFA1002300), the Youth Innovation Promotion Association of Chinese Academy Sciences, and the ``Chen Jingrun'' Future Star Program of AMSS. 

He would like to thank Florian Kranhold for pointing out a minor gap in an earlier proof of Theorem \ref{decomthm}, and is indebted to referees for many valuable comments which have improved the exposition of the paper.
 



\section{Rank $(n+1)$ bundles over $4$-manifolds} 
\label{sec: prelim}
In this section, we discuss necessary knowledge of rank $(n+1)$ vector bundles over simply connected $4$-manifolds with $n\geq 2$ from the perspective of homotopy theory. The argument extends that in \cite[Section 2]{Hua22} for the case $n=2$, with a revised exposition. To ensure clarity, we provide comprehensive details.

Let $N$ be a simply connected closed $4$-manifold such that $H^2(N;\mathbb{Z})\cong \mathbb{Z}^{\oplus d}$ with $d\geq 0$. 
There is the homotopy cofiber sequence
\begin{equation}\label{Ncofibreeq}
S^3\stackrel{\phi}{\longrightarrow} \bigvee_{i=1}^{d} S^2 \stackrel{\rho}{\longrightarrow} N \stackrel{q}{\longrightarrow} S^4\stackrel{\Sigma\phi}{\longrightarrow} \bigvee_{i=1}^{d} S^3,
\end{equation}
where $\phi$ is the attaching map of the top cell of $N$, $\rho$ is the inclusion of the $2$-skeleton, and $q$ is the pinch map onto the top cell. 
Since $\pi_3(BSO(n+1))=0$, it is easy to see that the homotopy cofibre sequence (\ref{Ncofibreeq}) implies the short exact sequence of pointed sets
\begin{equation}\label{seseq}
0\stackrel{}{\longrightarrow}[S^4, BSO(n+1)] \stackrel{q^\ast}{\longrightarrow} [N, BSO(n+1)] \stackrel{\rho^\ast}{\longrightarrow} [\bigvee_{i=1}^{d} S^2, BSO(n+1)] \stackrel{}{\longrightarrow} 0,
\end{equation}
in a strong sense that, there is an action of $[S^4, BSO(n+1)]$ on $[N, BSO(n+1)]$ through $q^\ast$ such that the sets $(\rho^\ast)^{-1} (x)$, for $x\in [\mathop{\bigvee}\limits_{i=1}^{d} S^2, BSO(n+1)]$ are precisely the orbits. Further, the short exact sequence is natural with respect to degree~$1$ maps between $4$-manifolds. That is, given a degree~$1$ map $g: N \to M$ between simply connected closed $4$-manifolds, it induces a morphism from the short exact sequence~\eqref{seseq} associated with $M$ to that associated with $N$.

A rank $(n+1)$ vector bundle $\xi$ over $N$ is classified by a map $f: N\longrightarrow BSO(n+1)$. The following lemma characterizes the restriction $\rho^\ast(f)$ by a circle bundle.
Let $s: S^1\cong SO(2)\rightarrow SO(n+1)$ be the canonical inclusion of Lie groups.
A cohomology class \( z \in H^2(N; \mathbb{Z}) \) is called 
\emph{primitive} if it cannot be written in the form 
\( z = k w \), where \( w \in H^2(N; \mathbb{Z}) \) and 
\( k \in \mathbb{Z} \) is a non-unit integer.

\begin{lemma}\label{alphalemma}
There exists a class $\alpha\in H^2(N;\mathbb{Z})\cong [N, BS^1]$, such that $\rho^\ast(f)=(\rho^\ast \circ (Bs)_\ast)(\alpha)$ through the composition
\[
[N, BS^1]\stackrel{(Bs)_\ast}{\longrightarrow} [N, BSO(n+1)] \stackrel{\rho^\ast}{\longrightarrow} [\bigvee_{i=1}^{d} S^2, BSO(n+1)],
 \]
and $\alpha$ satisfies that
\begin{itemize}
\item $\alpha=0$ if the second Stiefel-Whitney class $\omega_2(\xi)= 0$;
\item $\alpha$ is primitive and $\omega_2(\xi)\equiv \alpha~{\rm mod}~2$ if $\omega_2(\xi)\neq 0$.
\end{itemize}
\end{lemma}
\begin{proof}
Consider the diagram
\[
\label{w2diag}
\xymatrix{
[N, BS^1] \ar[r]^<<<<<<<<{\rho^\ast}_<<<<<<<<{\cong} \ar[d]^{(Bs)_\ast} & [\mathop{\bigvee}\limits_{i=1}^{d} S^2, BS^1] \ar[d]^{(Bs)_\ast}  \ar[r]^<<<<<<{e}_<<<<<<{\cong}   & H^2(\mathop{\bigvee}\limits_{i=1}^{d} S^2;\mathbb{Z})\cong \mathop{\bigoplus}\limits_{i=1}^{d} \mathbb{Z} \ar[d]^{\rho_2} \\
[N, BSO(n+1)] \ar[r]^<<<<{\rho^\ast}  & [\mathop{\bigvee}\limits_{i=1}^{d} S^2, BSO(n+1)]   \ar[r]^<<<{\omega_2}_<<<{\cong}&H^2(\mathop{\bigvee}\limits_{i=1}^{d} S^2;\mathbb{Z}/2) \cong \mathop{\bigoplus}\limits_{i=1}^{d} \mathbb{Z}/2,
}
\]
where the pullback $\rho^\ast$ on the top row, the Euler class morphism $e$ and the second Stiefel-Whitney class morphism $\omega_2$ are isomorphisms, and $\rho_2$ is the mod $2$ reduction. The left square commutes automatically, while the right square commutes as the mod $2$ reduction of the Euler class of a complex line bundle is its second Stiefel-Whitney class. Note that $(\omega_2\circ \rho^\ast) (f)=\rho^\ast (\omega_2(\xi))$, and the latter $\rho^\ast: H^2(N;\mathbb{Z}/2)\larrow H^2(\mathop{\bigvee}\limits_{i=1}^{d} S^2;\mathbb{Z}/2)$ is an isomorphism on the second cohomology. 

Suppose that $\omega_2(\xi)= 0$. Then $(\omega_2\circ \rho^\ast)(f)=\rho^\ast (\omega_2(\xi))=0$. It follows that $\rho^\ast(f)=0$ as $\omega_2$ is an isomorphism. Choose $\alpha=0$. Then $\rho^\ast(f)=0=(\rho^\ast \circ (Bs)_\ast)(\alpha)$.

Otherwise, suppose that $\omega_2(\xi)\neq 0$. Then there is a primitive element $x\in H^2(\mathop{\bigvee}\limits_{i=1}^{d} S^2;\mathbb{Z})$ such that $\rho_2(x)=(\omega_2\circ \rho^\ast)(f)=\rho^\ast (\omega_2(\xi))\neq 0$. 
Since $e$ and the top $\rho^\ast$ are isomorphisms in the above diagram, we may choose $\alpha=((\rho^\ast)^{-1}\circ e^{-1})(x)$. Then $\alpha$ is primitive, and 
\[
\begin{split}
(\rho^\ast \circ (Bs)_\ast)(\alpha)&= ((Bs)_\ast\circ \rho^\ast)\circ ((\rho^\ast)^{-1}\circ e^{-1})(x)\\
&=((Bs)_\ast\circ  e^{-1}) (x)=(\omega_2^{-1}\circ \omega_2 \circ (Bs)_\ast\circ  e^{-1}) (x) \\
&=(\omega_2^{-1}\circ \rho_2\circ e\circ  e^{-1}) (x)=(\omega_2^{-1}\circ \rho_2) (x)\\
&=\rho^\ast(f).
\end{split}
\]
Further, since the Euler class morphism $e$ and the mod $2$ reduction $\rho_2$ are natural, there is the commutative diagram
\[
\xymatrix{
[N, BS^1]   \ar[r]^{e}_{\cong}  \ar[d]_{\rho^\ast}^{\cong}  & H^2(N;\mathbb{Z}) \ar[r]^{\rho_2}   \ar[d]_{\rho^\ast}^{\cong} &  H^2(N;\mathbb{Z}/2)\ar[d]_{\rho^\ast}^{\cong}\\
 [\mathop{\bigvee}\limits_{i=1}^{d} S^2, BS^1]  \ar[r]^<<<{e}_<<<{\cong}   & H^2(\mathop{\bigvee}\limits_{i=1}^{d} S^2;\mathbb{Z}) \ar[r]^{\rho_2}    &  H^2(\mathop{\bigvee}\limits_{i=1}^{d} S^2;\mathbb{Z}/2),
}
\]
where all the morphisms except $\rho_2$ are isomorphisms. Then 
\[
(\rho_2\circ  e)(\alpha)=((\rho^\ast)^{-1}\circ \rho^\ast\circ \rho_2\circ  e)(\alpha)=((\rho^\ast)^{-1}\circ \rho_2\circ  e\circ \rho^\ast)(\alpha)=((\rho^\ast)^{-1}\circ \rho_2)(x)= \omega_2(\xi).
\]
\end{proof}

The following lemma is a general form of \cite[Lemma 2.1]{Hua22}. It decomposes the classifying map $f$ into two simpler parts through the action $\cdot$ defined by $q^\ast$ in (\ref{seseq}). 
\begin{lemma}\label{replemma}
Define a morphism $\Phi$ as the composite 
\[
\Phi: [S^4, BSO(n+1)]\times  [N, BS^1]\stackrel{({\rm id}, (Bs)_\ast)}{\lllarrow} [S^4, BSO(n+1)]\times [N,BSO(n+1)]\stackrel{\cdot}{\longrightarrow} [N, BSO(n+1)].
\]
Then $\Phi$ is surjective. Moreover, for any choice of $\alpha$ in Lemma \ref{alphalemma}, there exists a unique class $f^\prime\in  [S^4, BSO(n+1)]$ such that 
\[
\Phi(f^\prime, \alpha)=q^\ast(f^\prime)\cdot  (Bs_\ast(\alpha)) =f.
\]
\end{lemma}
\begin{proof}
Let $\alpha\in [N, BS^1]$ be a class in Lemma \ref{alphalemma} for the classifying map $f$. In particular, $\rho^\ast(f)=\rho^\ast((Bs)_\ast(\alpha))$. By the exactness of the sequence \eqref{seseq}, it follows that $f$ and $(Bs)_\ast(\alpha)$ are in the same orbit, and there is a uniqe class $f^\prime\in  [S^4, BSO(n+1)]$ such that $
q^\ast(f^\prime)\cdot  (Bs_\ast(\alpha)) =f
$. Since $f$ is arbitrary, this implies that the composite $\Phi$ is surjective.
\end{proof}

\begin{remark}
From the perspective of manifold topology, it is known that a rank $(n+1)$ vector bundle $\xi$ over $N$ is classified by its second Stiefel-Whitney class $\omega_2(\xi)$ and its first Pontryagin class $p_1(\xi)$, plus the extra Euler class $e(\xi)$ when $n=3$. 
This is consistent with the homotopy arguments in this section. Indeed, Lemma \ref{replemma} implies that the bundle $\xi$ is determined by a pair $(f^\prime, \alpha)$, where $\alpha$ is an integral lifting of $\omega_2(\xi)$ in Lemma \ref{alphalemma}, and $f'$ together with its action corresponds to the first Pontryagin class $p_1(\xi)$, plus the Euler class $e(\xi)$ when $n=3$. For example, when $\xi$ is trivial on the $2$-skeleton of $N$, the lifting $\alpha=0$ and the characteristic classes $p_1(\xi)$ and $e(\xi)$ are the pullback of those of the bundle over $S^4$ determined by $f'\in [S^4, BSO(n+1)]$.
\end{remark}


\section{The case when $d\geq 2$} 
\label{sec: d>1}
Let $N$ be a simply connected $4$-manifold such that $H^2(N;\mathbb{Z})\cong \mathbb{Z}^{\oplus d}$ with $d\geq 2$. A rank $(n+1)$ vector bundle $\xi$ over $N$ is classified by a map $f: N\longrightarrow BSO(n+1)$. 
Fix a class $\alpha\in H^2(N;\mathbb{Z})$ in Lemma \ref{alphalemma}.

\begin{lemma}\label{Qlemma}
There exists a simply connected $4$-dimensional Poincar\'{e} duality complex $Q=(S^2\vee S^2)\cup e^4$ with a degree $1$ map
\[
h: N\stackrel{}{\longrightarrow} Q,
\]
such that $h^\ast(x)=\alpha$ for some class $x\in H^2(Q;\mathbb{Z})$.
\end{lemma}
\begin{proof}
We consider the case when $\omega_2(\xi)\neq 0$ first. By Lemma \ref{alphalemma} the class $\alpha\in H^2(N;\mathbb{Z})$ is primitive.
We claim that there exists a class $\beta\in H^2(N;\mathbb{Z})$ such that $\beta \neq \alpha$, and $\langle \alpha\cup \beta, [N]\rangle =\pm 1$. 
Otherwise, by Poincar\'{e} duality $\langle \alpha\cup \alpha, [N]\rangle =\pm 1$. Since $d\geq 2$, we can choose another generator $\gamma\in H^2(N;\mathbb{Z})$. By hypothesis, $\langle \alpha\cup \gamma, [N]\rangle =\pm k$ with $k\neq 1$. If $k=0$, let $\beta= \alpha+\gamma$, and $\langle \alpha\cup (\alpha+\gamma), [N]\rangle =\pm 1$. If $k\geq 2$, let $\beta= (1-k)\alpha+\gamma$, and $\langle \alpha\cup ((1-k)\alpha+\gamma), [N]\rangle =\pm 1$. Hence, the claim is proved.

Then as in \cite[Proof of proposition 3.2]{BT14}, by abuse of notation the inclusion $\rho$ of the $2$-skeleton of $N$ can be chosen as
\[
\rho: K\vee (S^2\vee S^2) \stackrel{j\vee (\alpha\vee \beta)}{\longrightarrow} N,
\]
where $K=\mathop{\bigvee}\limits_{i=1}^{d-2}S^2$, and the last two $S^2$ represent the classes $\alpha$ and $\beta$. Let $Q$ be the homotopy cofibre of the composition $K\stackrel{}{\hookrightarrow}K\vee (S^2\vee S^2)\stackrel{\rho}{\longrightarrow} N$. Denote by $h: N\stackrel{}{\longrightarrow}Q$ the quotient map.
It is clear that $Q\simeq (S^2\vee S^2)\cup e^4$ as $CW$-complexes. By construction $h^\ast: H^\ast(Q;\mathbb{Z})\longrightarrow H^\ast(N;\mathbb{Z})$ is an isomorphism onto the submodule $\mathbb{Z}\alpha\oplus \mathbb{Z}\beta\subseteq H^2$ and an isomorphism on $H^4$, and 
\[
H^\ast(Q;\mathbb{Z})\cong H^\ast(S^2\times S^2;\mathbb{Z})\cong \mathbb{Z}[\alpha, \beta]/(\alpha^2, \beta^2).
\]
In particular, $Q$ is a Poincar\'{e} duality complex. By choosing $x=(h^{\ast})^{-1}(\alpha)$, the lemma is proved when $\omega_2(\xi)\neq 0$.

When $\omega_2(\xi)= 0$, the class $\alpha=0$ by Lemma \ref{alphalemma}. Choose any primitive class $\alpha^\prime \in H^2(N;\mathbb{Z})$. Then we can run through the previous argument for $\alpha^\prime$ instead of $\alpha$, except in the last step where we can simply let $x=0$. This proves the lemma when $\omega_2(\xi)= 0$.
\end{proof}

\begin{lemma}\label{loopQlemma}
For the Poincar\'{e} duality complex $Q$ in Lemma \ref{Qlemma}, there is a homotopy equivalence
\[
\Omega Q\simeq \Omega (S^2\times S^2),
\]
\end{lemma}
\begin{proof}
The lemma is a special case of \cite[Lemma 2.3]{BT14}. Consider the composite
\[
\Omega S^2\times \Omega S^2 \stackrel{\phi}{\larrow} \Omega(S^2\vee S^2)\stackrel{\Omega i}{\larrow} \Omega Q,
\]
where $\phi$ is a right homotopy inverse of the looped inclusion $\Omega (S^2\vee S^2)\longrightarrow \Omega (S^2\times S^2)$ by the Hilton-Milnor theorem, and $i$ is the inclusion of the $2$-skeleton of $Q$. By the proof of Lemma \ref{Qlemma}, we have the ring isomorphism $H^\ast(Q;\mathbb{Z})\cong H^\ast(S^2\times S^2;\mathbb{Z})$. Then the composite induces an isomorphism on homology and hence is a homotopy equivalence by the Whitehead theorem. 
\end{proof}

\begin{lemma}\label{fQlemma}
The classifying map $f$ of $\xi$ factors as 
\[
f: N\stackrel{h}{\longrightarrow} Q\stackrel{f_Q}{\longrightarrow} BSO(n+1),
\]
for some map $f_Q$.
\end{lemma}
\begin{proof}
It is clear that the degree $1$ map $h$ induces a diagram of homotopy cofibrations
\[
\xymatrix{
K\vee (S^2\vee S^2)  \ar[r]^<<<{\rho} \ar[d]^{h_\prime}  &
N \ar[r]^{q} \ar[d]^{h} &
S^4 \ar@{=}[d] \\
S^2\vee S^2 \ar[r]^<<<<<{\varrho}&
Q\ar[r]^{\mathfrak{q}} &
S^4,
}
\]
where the maps $\varrho$ and $\mathfrak{q}$ are the obvious inclusion and quotient maps, and the projection $h_\prime$ is the restriction of $h$ on the $2$-skeletons. It implies the morphism of short exact sequences
\[
 \xymatrix{
0 \ar[r] &  [S^4, BSO(n+1)]   \ar[r]^{\mathfrak{q}^\ast} \ar[d]^{}_{\cong}  &[Q, BSO(n+1)] \ar[r]^<<<<<{\varrho^\ast} \ar[d]^{h^\ast} & [S^2\vee S^2, BSO(n+1)] \ar[r] \ar[d]^{h_\prime^\ast}  & 0\\
0 \ar[r] &  [S^4, BSO(n+1)]   \ar[r]^{q^\ast}  &[N, BSO(n+1)] \ar[r]^<<<<{\rho^\ast} & [K\vee (S^2\vee S^2), BSO(n+1)] \ar[r]  & 0,
}
\]
where the rows are short exact sequences from (\ref{seseq}). 
By Lemma \ref{replemma}, we know that $f=\Phi(f^\prime, \alpha)=q^\ast(f^\prime)\cdot  (Bs_\ast(\alpha))$ for some $f^\prime\in  [S^4, BSO(n+1)]$.
Let $f_Q=\mathfrak{q}^\ast(f^\prime)\cdot  (Bs_\ast(x))$ with $x\in [Q, BS^1]$ in Lemma \ref{Qlemma}. 
Then by Lemma \ref{Qlemma} and the naturality of the actions through the above diagram \cite{Swi02, Ark11}, we have
\[
\begin{split}
h^\ast(f_Q)
&=h^\ast(\mathfrak{q}^\ast(f^\prime)\cdot  (Bs_\ast(x)))\\
&=q^\ast(f^\prime) \cdot  (Bs_\ast(h^\ast (x)))\\
&=q^\ast(f^\prime) \cdot  (Bs_\ast(\alpha))\\
&=f.
\end{split}
\]
This proves the lemma.
\end{proof}

We are ready to prove Theorem \ref{decomthm} when $d\geq 2$.
\begin{proof}[Proof of Theorem \ref{decomthm} when $d\geq 2$]
In this case, by Lemmas \ref{Qlemma} and \ref{fQlemma}, there exists a Poincar\'{e} duality complex $Q=(S^2\vee S^2)\cup e^4$ such that the classifying map $f$ of $\xi$ factors as 
\[
f: N\stackrel{h}{\longrightarrow} Q\stackrel{f_Q}{\longrightarrow} BSO(n+1)
\] 
for a degree one map $h$ and a map $f_Q$. The latter map $f_Q$ determines a vector bundle $\xi_Q$ over $Q$, the pullback of which along $h$ is isomorphic to $\xi$. This implies a morphism of sphere bundles
\begin{equation}
\label{NHQbundeleq}
\begin{gathered}
 \xymatrix{
S^n\ar@{=}[d] \ar[r]^{i} &
M \ar[r]^{p} \ar[d]^{\widetilde{h}}&
N\ar[d]^{h}\\
S^n \ar[r]^{\widetilde{i}} &
\widetilde{Q} \ar[r]^{\widetilde{p}} &
Q,
}
\end{gathered}
\end{equation}
where $i$ and $\widetilde{i}$ are fibre inclusions, $p$ and $\widetilde{p}$ are bundle projections, 
$\widetilde{h}$ is the induced map, and $\widetilde{Q}$ is an $(n+4)$ dimensional Poincar\'{e} duality complex.

By Lemma \ref{loopQlemma}, $\Omega Q\simeq \Omega (S^2\times S^2)$. Since $n\geq 2$, the cell structures of $Q$ and $\widetilde{Q}$ imply that the projection $\widetilde{p}$ has a partial section $i_1\vee i_2: S^{2}\vee S^{2}\rightarrow \widetilde{Q}$ on the $2$-skeleton $S^{2}\vee S^{2}$ of $Q$. By the Hilton-Milnor theorem, the canonical inclusion $S^{2}\vee S^{2}\stackrel{j}{\longrightarrow} S^{2}\times S^{2}$ has a right homotopy inverse after looping, that is, there is a map $\phi: \Omega S^{2}\times \Omega S^{2}\stackrel{}{\longrightarrow}\Omega(S^{2}\vee S^{2})$ such that $\Omega j\circ\phi\simeq {\rm id}$. Then it is easy to see that the composite
\[
\Omega S^{2}\times \Omega S^{2}\stackrel{\phi}{\longrightarrow}\Omega(S^{2}\vee S^{2}) \stackrel{\Omega (i_1\vee i_2)}{\llarrow} \Omega \widetilde{Q}\stackrel{\Omega\widetilde{p}}{\longrightarrow}\Omega Q
\]
induces an isomorphism in homology, and so is a homotopy equivalence by the Whitehead Theorem. 
Therefore, the looped map $\Omega \tilde{p}$ admits a right homotopy inverse, and the sphere bundle of $\xi_Q$ splits after looping to give
\[
\Omega \widetilde{Q}\simeq \Omega S^n\times \Omega Q\simeq  \Omega S^n\times\Omega S^{2}\times \Omega S^{2}.
\]
In particular, in Diagram (\ref{NHQbundeleq}) $\Omega \widetilde{i}$ has a left homotopy inverse $\widetilde{r}$, which implies that $\widetilde{r}\circ \Omega\widetilde{h}$ is a left homotopy inverse of $\Omega i$. 
Then the sphere bundle in the top row of Diagram (\ref{NHQbundeleq}) splits after looping to give
\[
\Omega M\simeq \Omega S^n\times \Omega N.
\]
Further, by \cite[Theorem 1.3]{BT14} when $d\geq 2$ there is a homotopy equivalence
\[
 \Omega N\simeq S^1\times \Omega (S^2\times S^3)\times  \Omega\big(J\vee(J\wedge\Omega (S^2\times S^3))\big)
\]
with $J=\mathop{\bigvee}\limits_{i=1}^{d-2}(S^2\vee S^3)$. Then in this case the theorem follows by combining the above decompositions.
\end{proof}


\section{The case when $d=1$} 
\label{sec: d=1}
Let $N$ be a simply connected $4$-manifold such that $H^2(N;\mathbb{Z})\cong \mathbb{Z}$.
\begin{lemma}\label{hCP2lemma}
There is a homotopy equivalence $N\simeq \mathbb{C}P^2$.
\end{lemma}
\begin{proof}
The generator of $H^2(N;\mathbb{Z})\cong \mathbb{Z}$ is represented by a map $N\rightarrow \mathbb{C}P^\infty$. By the cellular approximation theorem it factors through $\mathbb{C}P^2$ as
\[
N\stackrel{\lambda}{\longrightarrow}\mathbb{C}P^2\stackrel{}{\hookrightarrow}\mathbb{C}P^\infty
\]
for some map $\lambda$.
Then by Ponincar\'{e} duality, $\lambda$ induces an isomorphism on cohomology and then is a homotopy equivalence by the Whitehead theorem.
\end{proof}
By Lemma \ref{hCP2lemma}, $N$ is homotopy equivalent to $\mathbb{C}P^2$. Since we are only interested in homotopy types, we can simply let $N=\mathbb{C}P^2$.

A rank $(n+1)$ vector bundle $\xi$ over $\mathbb{C}P^2$ is classified by a map $f: \mathbb{C}P^2\longrightarrow BSO(n+1)$. Consider the diagram of fibre bundles
\begin{equation}\label{keydiag}
\begin{gathered}
 \xymatrix{
&
S^n \ar@{=}[r] \ar[d]^{\iota} &
S^n \ar[d]^{i} \\
S^1\ar@{=}[d] \ar[r]^{\jmath} &
X \ar[d]^{\mathfrak{p}} \ar[r]^{\psi} &
M \ar[d]^{p}\\
S^1 \ar[r]^>>>>{j} &
S^5 \ar[r]^<<<<{\pi} &
\mathbb{C}P^2,
}
\end{gathered}
\end{equation}
where the bottom bundle is the canonical circle bundle of $\mathbb{C}P^2$, the rightmost bundle is the sphere bundle of $\xi$, and $X$ is a closed $(n+5)$-manifold.
\begin{lemma}\label{pixilemma}
The pullback vector bundle $\pi^\ast(\xi)$ is trivial, and in particular its sphere bundle
\[
X\cong S^n\times S^5.
\]
\end{lemma}
\begin{proof}
When $n\geq 5$, by Bott periodicity $\pi_5(BSO(n+1))\cong \pi_4(SO(n+1))=0$, and then the lemma follows in this case.
When $2\leq n\leq 4$, consider the diagram of homotopy cofibrations
\[
\xymatrix{
\ast  \ar[r]^<<<{} \ar[d]^{}  &
S^5 \ar@{=}[r] \ar[d]^{\pi} &
S^5   \ar[d]^{\pi^\prime}\\
S^2\ar[r]^{\rho}&
\mathbb{C}P^2\ar[r]^{q} &
S^4,
}
\]
where $\pi^\prime=q\circ \pi$. From (\ref{seseq}) it implies the morphism of short exact sequence
\[
 \xymatrix{
0 \ar[r] &  [S^4, BSO(n+1)]   \ar[r]^{q^\ast} \ar[d]^{(\pi^{\prime})^{\ast}}  &[\mathbb{C}P^2, BSO(n+1)] \ar[r]^{\rho^\ast} \ar[d]^{\pi^\ast} & [S^2, BSO(n+1)] \ar[r]  & 0\\
&  [S^5, BSO(n+1)]   \ar@{=}[r]^{}  &[S^5, BSO(n+1)].
}
\]
It is known that the homotopy cofibre of $\pi$ is $\mathbb{C}P^3$, and the Steenrod operation $Sq^2: H^4(\mathbb{C}P^3;\mathbb{Z}/2)\rightarrow H^6(\mathbb{C}P^3;\mathbb{Z}/2)$ is trivial. This is equivalent to that $\pi^\prime$ is null homotopic. In particular $(\pi^{\prime})^{\ast}=0$.

By Lemma \ref{replemma}, we know that $f=\Phi(f^\prime, \alpha)=q^\ast(f^\prime)\cdot  (Bs_\ast(\alpha))$ for some $f^\prime\in  [S^4, BSO(n+1)]$ and $\alpha\in [\mathbb{C}P^2, BS^1]$. Notice that $\pi^\ast (\alpha)=0$ as $H^2(S^5;\mathbb{Z})=0$.
Then by the naturality of the actions through the above diagram \cite{Swi02, Ark11}
\[
\begin{split}
\pi^\ast(f)
&=\pi^\ast(q^\ast(f^\prime)\cdot  (Bs_\ast(\alpha)))\\
&=(\pi^{\prime})^{\ast}(f^\prime) \cdot  (Bs_\ast(\pi^\ast (\alpha)))\\
&=0.
\end{split}
\]
It means that $\pi^\ast(\xi)$ is trivial and the lemma is proved.
\end{proof}

\begin{remark*}
Note the above proof is valid for any $G$ such that $\pi_2(G)=0$, which is the case for any compact Lie group. Hence, for a such $G$ we have showed for any principal $G$-bundle over $\mathbb{C}P^2$, its pullback along the standard projection $S^5\stackrel{}{\rightarrow}\mathbb{C}P^2$ is trivial.
\end{remark*}

We are ready to prove Theorem \ref{decomthm} when $d=1$.
\begin{proof}[Proof of Theorem \ref{decomthm} when $d=1$]
In this case, consider the circle bundle 
\[
S^1\stackrel{\jmath}{\longrightarrow} X\stackrel{\psi}{\longrightarrow} M
\]
in the middle row of Diagram (\ref{keydiag}). Since $X$ is simply connected, $\jmath$ is null homotopic. Therefore, $\Omega \psi$ has a left homotopy inverse $\Omega M\stackrel{r}{\rightarrow} \Omega X$, and then $\Omega M\simeq S^1\times \Omega X$. Since by Lemma \ref{pixilemma} $X\cong S^n \times S^5$, we can define the composition
\[
\Omega M\stackrel{r}{\longrightarrow} \Omega X\cong  \Omega (S^n \times S^5) \stackrel{q_1}{\longrightarrow} \Omega S^n,
\]
where $q_1$ is the obvious projection. By Diagram (\ref{keydiag}) we see that the composition is a left homotopy inverse of $\Omega S^n\stackrel{\Omega i}{\longrightarrow} \Omega M$. It follows that the sphere bundle in the rightmost column of Diagram (\ref{keydiag}) splits after looping to give a homotopy equivalence
\[
\Omega M\simeq \Omega S^n\times \Omega \mathbb{C}P^2 \simeq S^1 \times \Omega S^n\times \Omega S^5.\]
This completes the proof of the theorem when $d=1$.
\end{proof}


\section{The case when $d=0$} 
\label{sec: d=0}
Let $N$ be a simply connected $4$-manifold such that $H^2(N;\mathbb{Z})\cong 0$. Then $N$ is a homotopy equivalent to $S^4$, and we may let $N=S^4$ as we are only interested in their homotopy types. 
\begin{proof}[Proof of Theorem \ref{decomthm} when $d=0$]
In this case, the sphere bundle is $S^n \stackrel{i}{\longrightarrow} M\stackrel{p}{\longrightarrow} S^4$ with $n\geq 4$.
It induces the long exact sequence of homotopy groups
\[
\cdots \stackrel{}{\longrightarrow} \pi_4(M)\stackrel{p_\ast}{\longrightarrow} \pi_4(S^4)\cong \mathbb{Z}\stackrel{}{\longrightarrow} \pi_3(S^n)=0\stackrel{}{\longrightarrow} \cdots.
\]
In particular, $p^\ast$ is surjective, and then $p$ has a right homotopy inverse. It follows the sphere bundle splits after looping to give a homotopy equivalence
\[
\Omega M\simeq \Omega S^n\times \Omega S^4 \simeq S^3\times \Omega S^n\times \Omega S^7.\]
This completes the proof of the theorem when $d=0$.
\end{proof}

There are two cases which are excluded from the theorem: $(d, n)=(0, 2)$ or $(0, 3)$. Let us discuss them briefly.

When $(d, n)=(0, 3)$, we have the sphere bundle $S^3 \stackrel{i}{\longrightarrow} M\stackrel{p}{\longrightarrow} S^4$. Following the notation of Crowley-Escher \cite{CE03}, the total manifold $M$ is denoted by $M_{\rho, \gamma}$ if it is classified by $\rho z+\gamma w\in \mathbb{Z}\{z, w\}\cong \pi_3(SO(4))$, where the two generators $z$ and $w$ are defined by 
$
z(u)v=uvu^{-1}
$ 
and $w(u)v=uv$ for any $u$, $v\in S^3\cong Sp(1)$.

 Moreover, as cell complex 
\[
M_{\rho, \gamma}\simeq P^4(\gamma)\cup_{}e^7,
\]
where the Moore space $P^4(\gamma)$ is the mapping cone of the degree map $\gamma: S^3\rightarrow S^3$. In particular, if $\gamma=1$, $M_{\rho, 1}$ is a homotopy sphere. 

If $\gamma=0$, $M_{\rho, 0}\simeq (S^3\vee S^4)\cup_{}e^7$. By \cite[Section 26.6]{Ste99}, the bundle $S^3 \stackrel{i}{\longrightarrow} M_{\rho, 0}\stackrel{p}{\longrightarrow} S^4$ has a cross section, and hence it splits after looping to give a homotopy equivalence
\[
\Omega M_{\rho, 0}\cong \Omega S^3\times \Omega S^4.
\]

If $\gamma\geq 3$ and is odd, the loop decomposition of the Poincar\'{e} duality complex $P^4(\gamma)\cup e^7$ was determined by Theriault and the author \cite{HT22}. For any prime $p$, let $S^{m}\{p^r\}$ be the homotopy fibre of the degree $p^r$ map on $S^{m}$.
Let $\gamma=p_1^{r_1}\cdots p_\ell^{r_\ell}$ be the prime decomposition of $\gamma$. By \cite[Theorem 1.1]{HT22}, when $\gamma$ is odd there is a homotopy equivalence 
\[\label{HTthmeq}
\Omega M_{\rho, \gamma}\simeq \Omega(P^4(\gamma)\cup e^7)\simeq \prod_{j=1}^{\ell} S^3\{p_j^{r_j}\}\times \Omega S^7.
\]

If $\gamma\geq 2$ and is even, it is a hard problem to determine the loop decomposition of $M_{\rho, \gamma}$. Partial results are obtained in \cite[Section 6]{HT22}.

When $(d, n)=(0, 2)$, we have the sphere bundle $S^2 \stackrel{i}{\longrightarrow} M\stackrel{p}{\longrightarrow} S^4$. This case was studied in \cite[Section 5]{Hua22}. 
Let $x\in H^2(M;\mathbb{Z})$, $y\in H^4(M;\mathbb{Z})$ be two generators. Suppose
$
x^2=\gamma y
$
for some $\gamma\in \mathbb{Z}$.
Then it can be shown that \cite[Lemma 5.1]{Hua22} the circle bundle classified by $x$ is homotopy equivalent to a homotopy fibration 
\[\label{Xd=0def}
S^1\stackrel{j}{\longrightarrow} M_{\rho, \gamma} \stackrel{}{\longrightarrow} M,
\]
where $j$ is null homotopic.
Therefore, there is a homotopy equivalence
\[
\Omega M\simeq S^1\times \Omega M_{\rho, \gamma}.
\]
This case is then reduced to the case when $(d, n)=(0, 3)$, and with the results of \cite{HT22} we can prove a nice loop homotopy decomposition when $\gamma$ is odd, and partial results when $\gamma$ is even.

In the end of this section, we digress to give a counterexample demonstrating that $H_\ast(M)$ is not the tensor product $H_\ast(S^n)\otimes H_\ast(N)$.
\begin{example}\label{Hnonex}
Consider the pullback of the sphere bundle $M_{\rho,\gamma }$ along the quotient map $q$ 
\[
\begin{gathered}
 \xymatrix{
 S^3 \ar[r] \ar@{=}[d]  & M \ar[r] \ar[d] & \mathbb{C}P^2 \ar[d]^{q} \\
 S^3    \ar[r]              & M_{\rho,\gamma } \ar[r]^{} & S^4,
}
\end{gathered}
\]
which defines the manifold $M$. By the naturality of the Serre spectral sequence, it is easy to see that the homology of $M$ satisfies
\[
H_0(M;\mathbb{Z})\cong H_2(M;\mathbb{Z})\cong H_5(M;\mathbb{Z})\cong H_7(M;\mathbb{Z})\cong \mathbb{Z}, \ \ H_3(M;\mathbb{Z})\cong \mathbb{Z}/\gamma, \ \ H_i(M;\mathbb{Z})=0 ~{\rm for}~{\rm other}~i.
\]
In particular, it implies that as graded groups
\[
H_\ast(M;\mathbb{Z})\not\cong H_\ast(S^3;\mathbb{Z}) \otimes H_\ast(\mathbb{C}P^2;\mathbb{Z}).
\]
\end{example}


\bibliographystyle{amsalpha}

\end{sloppypar}
\end{document}